\documentclass{amsart}

\newcommand{\II}{\mathbb I}
\newcommand{\e}{\varepsilon}
\newcommand{\IR}{\mathbb R}
\newcommand{\nw}{\mathrm{nw}}

\newtheorem{theorem}{Theorem}
\newtheorem{problem}{Problem}
\newtheorem{corollary}{Corollary}
\newtheorem{example}{Example}
\newtheorem{claim}{Claim}

\title{On locally extremal functions on connected spaces}
\author{T.~Banakh, M.~Vovk, M.~W\'ojcik}
\address{T.Banakh: Ivan Franko Lviv National University, Lviv, Ukraine, and\newline
Unwersytet Humanistyczno-Przyrodniczy im. Jana Kochanowskiego, Kielce, Poland}
\email{tbanakh@yahoo.com}
\address{M.Vovk: National University ``Lvivska Politechnika", Lviv, Ukraine}
\address{M.R.W\'ojcik: Instytut Matematyki i Informatyki, Politechnika Wroclawska, Wroclaw, Poland}
\email{michal.r.wojcik@pwr.wroc.pl}
\subjclass{54D05; 54C30}

\begin{document}

\begin{abstract} We construct an example of a continuous non-constant function $f:X\to \IR$ defined on a connected complete metric space $X$ such that every point $x\in X$ is a point of local minimum or local maximum for $f$. Also we show that any such a space $X$ should have the network weight and the weak separate number $\nw(X)\ge R(X)\ge\mathfrak c$.
\end{abstract}
\maketitle

This note is motivated by the following problem of M.W\'ojcik \cite{Woj1}, \cite{MW} (see also \cite[Question 7]{BGN}) on the nature of locally extremal functions on connected spaces. We define a function $f:X\to \IR$ on a topological space $X$ to be {\em locally extremal} if each point $x\in X$ is a point of local minimum or local maximum for $f$.

\begin{problem} Is a each locally extremal continuous function $f:X\to \IR$ on a connected metric space $X$ constant?
\end{problem}

In \cite{BGN} this problem has been answered in affirmative under the condition that the density of the connected metric space $X$ is strictly smaller than $\mathfrak c$, the size of continuum. In fact, this result is a particular case a more general theorem treating locally extremal functions on connected topological spaces $X$ with the weak separate number $R(X)<\mathfrak c$. 

Following M.Tkachenko \cite{Tk}, we define a topological space $X$ to be {\em weakly separated} if to each point $x\in X$ has an open neighborhood $O_x\subset X$ such that for any two distinct points $x,y\in X$ either $x\notin O_y$ or $y\notin O_x$. The cardinal
$$R(X)=\sup\{|Y|:\mbox{$Y$ is a weakly separated subspace of $X$}\}$$
is called {\em the weak separate number} of $X$. By \cite{Tk}, $R(X)\le nw(X)\le w(X)$, where $w(X)$ (resp. $nw(X)$) stands for the (network) weight of $X$. On the other hand, A.~Hajnal and I.~Juhasz \cite{HJ} constructed a CH-example of a regular space $X$ with $\aleph_0=R(X)<nw(X)=\mathfrak c$. It is an open problem if such an example exists in ZFC, see Problem 15 in \cite{GM}.

\begin{theorem}\label{t1} If $X$ is a topological space and $f:X\to \IR$ is a locally extremal function, then $|f(X)|\le R(X)$. 
\end{theorem}

\begin{proof} Write $X$ as the union $X=X_0\cup X_1$ of the sets $X_0$ and $X_1$ consisting of local minimums and local maximums of the function $f$, respectively. We claim that $|f(X_0)|\le R(X)$. Assuming the converse, find a subset $A\subset X_0$ such that $|A|>R(X)$ and $f|A$ is injective. Each point $a\in A$, being 
a point of local minimum of $f$, possesses a neighborhood $O_a\subset X$ such that $f(a)\le f(x)$ for all $x\in O_a$. We claim that the family of neighborhoods $\{O_a\}_{a\in A}$ witnesses that the set $A$ is weakly separated. Assuming the opposite, we would find two distinct points $a,b\in A$ such that $a\in O_b$ and $b\in O_a$. It follows from $b\in O_a$ that $f(a)\le f(b)$ and from $a\in O_b$ that $f(b)\le f(a)$. Consequently, $f(a)=f(b)$, which contradicts the injectivity of $f$ on $A$. This contradiction proved the inequality $|f(X_0)|\le R(X)$. By analogy we can prove that $|f(X_1)|\le R(X)$. 

If $R(X)$ is infinite then the inequality $\max\{|f(X_0)|,|f(X_1)|\}\le R(X)$ implies $|f(X)|\le|f(X_0)|+|f(X_1)|\le 2R(X)=R(X)$. Now assume that $n=R(X)$ is finite. We claim that $|f(X)|\le n$. Assuming the converse, we would find a finite subset $A=\{a_0,\dots,a_n\}\subset X$ such that $|f(A)|=n+1$.
The injectivity and continuity of the map $f|A$ guarantees that the subspace $A$ is discrete and hence weakly separated. Consequently, $R(X)\ge|A|=n+1>R(X)$, which is a contradiction.
\end{proof}

\begin{corollary}\label{c1} Each continuous locally extremal function $f:X\to \IR$ on a connected topological space $X$ with $R(X)<\mathfrak c$ is constant.
\end{corollary}

\begin{proof} By Theorem~\ref{t1}, $f(X)$ is a connected subset of $\IR$ with cardinality $|f(X)|\le R(X)<\mathfrak c$, which is possible only if $f(X)$ is a singleton.
\end{proof}

The condition $R(X)<\mathfrak c$ is essential in this corollary as the following example from \cite{MW} and \cite{BGN} shows.

\begin{example} The projection $f:[0,1]^2\to[0,1]$ of the lexicographic square onto the interval is locally extremal but not constant.
\end{example}

In \cite{DF} A.~Le Donne and A.~Fedeli announced the existence of a non-constant locally extremal continuous function defined on a connected metric space.
In the following theorem we describe a completely-metrizable example with the same properties.

\begin{theorem}\label{t2} There is a connected complete metric space and a continuous function $f:X\to\IR$ which is locally extremal but not constant.
\end{theorem}

\begin{proof} Let $\II^{\uparrow}=\{x^{\uparrow}:x\in\II\}$ and $\II^{\downarrow}=\{x^{\downarrow}:x\in\II\}$ be two disjoint copies of the unit interval $\II=[0,1]$. A basic building block of our metric space $X$ is the full graph 
$$T=\bigcup_{x,y\in \II^{\downarrow}\cup\II^{\uparrow}}[x,y]\subset l_1(\II^{\downarrow}\cup\II^{\uparrow})$$
with the set of vertices $\II^{\downarrow}\cup\II^{\uparrow}$ in the Banach space $$l_1(\II^{\downarrow}\cup\II^{\uparrow})=\{g:\II^{\downarrow}\cup\II^{\uparrow}\to\IR:\sum_{x\in \II^{\downarrow}\cup\II^{\uparrow}}|f(x)|<\infty\}.$$
For every $\alpha\in[0,1]$ the copy $T_\alpha=T\times\{\alpha\}$ of $T$ will be called the $\alpha$th town. The space $T_\alpha$ has diameter 2 in the metric $d_\alpha$  induced from $l_1$-metric of the Banach space $l_1(\II^{\downarrow}\cup\II^{\uparrow})$. It will be convenient to think of the distance in the town $T_\alpha$ as the smallest amount of time for getting from one place to another place of $T_\alpha$ by a taxi that moves with velocity 1.

The vertices  
$$
\begin{aligned}
&x_\alpha^{\downarrow}=(x^{\downarrow},\alpha)\in\II^{\downarrow}\times\{\alpha\}\subset T_\alpha,\\
&x_\alpha^{\uparrow}=(x^{\uparrow},\alpha)\in\II^{\uparrow}\times\{\alpha\}\subset T_\alpha
\end{aligned}
$$  
of the graph $T_\alpha$ are called lower and upper airports, respectively.
For any indices $\gamma<\alpha<\beta$ in $\II$, between the airports $\alpha^{\uparrow}_\alpha\in T_\alpha$ and 
$\alpha^{\downarrow}_\beta\in T_\beta$ there is an air connection taking $\beta-\alpha$ units of time.
Similarly, between the airports $\alpha^{\downarrow}_\alpha$ and $\alpha_\gamma^{\uparrow}$ there is an air connection taking $\alpha-\gamma$ units of time.

Now define a metric $d$ on the space $X=\bigcup_{\alpha\in\II}T_\alpha$ as the smallest amount of time necessary to get from one place to another place of $X$ using taxi (inside of the towns) and planes (between the towns).

More formally, this metric $d$ on $X$ can be defined as follows. In the square $X\times X$ consider the subset $$D=\bigcup_{\alpha\in\II}T_\alpha\times T_\alpha\cup 
\bigcup_{\alpha<\beta}\{(\alpha_\alpha^+,\alpha_\beta^-),(\alpha_\beta^-,\alpha_\alpha^+),(\beta_\beta^-,\beta_\alpha^+),(\beta_\alpha^+,\beta_\beta^-)\}$$
and define a function $\rho:D\to\IR$ letting
$$\rho(x,y)=
\begin{cases}
d_\alpha(x,y)&\mbox{if $x,y\in T_\alpha$, $\alpha\in \II$;}\\
|\beta-\alpha|&\mbox{if $\{x,y\}\in \big\{\{\alpha_\alpha^+,\alpha_\beta^-\},\{\beta_\beta^-,\beta_\alpha^+\}\big\}$ for some $\alpha<\beta$ in $\II$.}
\end{cases}
$$
This function induces a metric $d$ on $X$ defined by
$$d(x,y)=\inf\Big\{\sum_{i=1}^n\rho(x_{i-1},x_i):\forall i\le n\;(x_{i-1},x_i)\in D,\; x_0=x,\; x_n=y\Big\}.$$
It is easy to check that $d$ is a complete metric on $X$.

Next, define a map $f:X\to\II\subset\IR$ letting $f^{-1}(\alpha)=T_\alpha$ for $\alpha\in\II$. It is easy to see that this map is non-constant and non-expanding and hence continuous.

\begin{claim} Each point $x\in X$ is a point of local minimum or local maximum of the map $f:X\to\II$.
\end{claim}

\begin{proof} Take any point $x\in T_\alpha\subset X$, $\alpha\in\II$. If $x\notin\II^{\downarrow}\cup\II^{\uparrow}$ is not an airport, then we can find $\e>0$ such that the open $\e$-ball $O_\e(x)$ of $x$ in $T_\alpha$ does not intersects the set $\II_\alpha^{\downarrow}\cup\II_\alpha^{\uparrow}$. In this case $O_\e(x)$ is open in $X$ and hence $f$ is locally constant at $x$. 

Now assume that $x\in\II^{\downarrow}_\alpha\cup\II^{\uparrow}_\alpha$ and hence $x=\beta^{\downarrow}_\alpha$ or $\beta^{\uparrow}_\alpha$ for some $\beta\in\II$. First consider the case when $x=\beta^{\downarrow}_\alpha$ coincides with a lower airport. If $\beta=\alpha$, then we can consider the 1-neighborhood ball $O_1(x)$ around the airport $x=\alpha^{\downarrow}_\alpha$ and observe that this neighborhood does not intersect towns $T_\gamma$ with $\gamma>\alpha$. Consequently, $f(O_1(x))\subset(-\infty,\alpha]=(-\infty,f(x)]$, which means that $x$ is a point of local maximum of $f$.
If $\beta\ne\alpha$, then we consider the open ball $O_\e(x)\subset X$ of radius $\e=\beta-\alpha$ centered at $x=\beta^{\downarrow}_\alpha$ and observe that it contains only points of the town $T_\alpha$, which means that $f$ is locally constant at $x$.

The case $x=\II^{\uparrow}_\alpha$ can be considered by analogy.
\end{proof}

\begin{claim} The metric space $X$ is connected.
\end{claim}

\begin{proof}
Given a  non-empty open-and-closed subset $U\subset X$, we should prove that $U=X$. First we show that the image $f(U)$ is open. Given any point $\alpha\in f(U)$, fix any $x\in U\cap T_\alpha$. It follows from the connectedness of the town  $T_\alpha$ that $T_\alpha\subset U$. Consider the airports $\alpha_\alpha^{\downarrow}$, $\alpha^{\uparrow}_\alpha$ and find $\e>0$ such that the open set $U$ contains the open $\e$-balls $O_\e(\alpha_\alpha^{\downarrow})$ and $O_\e(\alpha_\alpha^{\uparrow})$ centered at those airports. It follows from the definition of the metric $d$ on $X$ that $$f(O_\e(\alpha_\alpha^{\downarrow}))\supset (\alpha-\e,\alpha]\cap\II\mbox{ and } f(O_\e(\alpha_\alpha^{\uparrow}))\supset [\alpha,\alpha+\e)\cap\II.$$ Unifying those inclusions, we get $f(U)\supset(\alpha-\e,\alpha+\e)\cap\II$, witnessing that $f(U)$ is open in $\II$. By the same reason $(X\setminus U)$ is open in $\II$. The connectedness of the fibers $T_\alpha$ of $f$ implies that the sets $f(U)$ and $f(X\setminus U)$ are disjoint. Now the connectedness of the interval $\II=f(U)\cup f(X\setminus U)$ guarantees that $f(X\setminus U)$ is empty and hence $U=X$.
\end{proof}
\end{proof}

The space $X$ from Theorem~\ref{t2} has an interesting property: it is connected but not separably connected (and hence not path connected). Following \cite{CHI} or \cite{BEHV}, we define a topological space $X$ to be {\em separably connected} if any two points $x,y\in X$ lie is a connected separable subspace of $X$. The problem on the existence of connected metric spaces that are not separably connected was posed in \cite{BEHV} and was answered in \cite{AM} and \cite{WPhD}. However all known examples of such spaces are nor complete.

\begin{theorem} There is a complete metric space which is connected but not separably connected.
\end{theorem}

\begin{proof} The connected complete metric space $X$ from Theorem~\ref{t2} is not separably connected by Corollary~\ref{c1}.
\end{proof}

The connected metric space $X$ from Theorem~\ref{t2} is not separably connected but contains many non-degenerate separable connected subspaces.

\begin{problem} Is there a connected  complete metric space $C$ such that each connected separable subspace of $C$ is a singleton.
\end{problem}

Non-complete connected metric spaces with this property were constructed in Theorem 28 of \cite{WPhD}.

\end{document}